\documentclass[reqno,a4paper,12pt]{amsart}
\usepackage{amsmath,amsthm,amsfonts,amssymb, mathrsfs}
\usepackage{verbatim, graphicx, ifthen, enumitem}
\usepackage[T1]{fontenc}
\usepackage [applemac] {inputenc}

\allowdisplaybreaks

\let\oldmarginpar\marginpar
\renewcommand\marginpar[1]{\-\oldmarginpar[\raggedleft\footnotesize #1]%
{\raggedright\footnotesize #1}}
\newtheorem{theorem}{Theorem}
\newtheorem{lemma}{Lemma}

\newtheorem{proposition}{Proposition}

\theoremstyle{definition}
\newtheorem{definition}{Definition}

\theoremstyle{remark}
\newtheorem{remark}{Remark}

\newcommand{\Z}{\mathbb{Z}}

\newcommand{\abs}[1]{|#1|}

\newcommand{\Bigabs}[1]{\Big|#1\Big|}

\newcommand{\R}{\mathbb{R}}

\def\Z{\mathbb{Z}}

\def\R{\mathbb{R}}

\def\1{\mathbf{1}}

\newcommand{\dif}{\mathrm{d}}
\newcommand{\e}{\mathrm{e}}
\newcommand{\im}{\mathrm{i}}
\newcommand{\norm}[1]{\|#1\|}

\author{Shahaf Nitzan}
\address{School of Mathematical Sciences, Weizmann Institute for Science, Rehovot 76100, Israel.} 
\email{shahaf.nitzansi@weizmann.ac.il}

\author{Jan-Fredrik Olsen}
\address{Centre for Mathematical Sciences, Lund University, P.O. Box 118, SE-221 00 Lund, Sweden}
\email{janfreol@maths.lth.se}


\begin{document}

 \title{A quantitative     Balian-Low theorem}

\begin{abstract}

We study   functions generating Gabor Riesz bases on the integer lattice.
The classical Balian-Low theorem restricts  the simultaneous time and frequency localization of such functions.
We obtain a quantitative estimate that extends both this result
and other related theorems.
\end{abstract}

\maketitle

\section{Introduction}

Let $g \in L^2(\R)$. The Gabor system generated by $g$ with
respect to the lattice $\Z \times \Z$ is the set
\begin{equation*}
    G(g) := \{ \e^{2\pi \im  n t} g(t - m)\}_{(m,n) \in \Z^2}.
\end{equation*}
Lattice Gabor systems play an important role in time-frequency
analysis and its applications \cite{dorfler2001, grochenig2001, strohmer2006}. A central
problem in this area is to describe the optimal
time-frequency localization of generators of "good" Gabor systems, e.g., orthonormal bases and Riesz bases.

The classical Balian-Low theorem \cite{balian1981, daubechies1990, low1985} is an uncertainty
principle for   generators of Gabor Riesz bases. It states that if $G(g)$ is such a system, then
\begin{equation}\label{clasical blt condition}
     \int_\R \abs{g(t)}^2 t^2 \dif t   =\infty \:\:\:\:\:\text{or}\:\:\:\:\: \int_\R \abs{\hat{g}(\xi)}^2 \xi^2\dif \xi  = \infty.
\end{equation}
In the last 20 years, additional versions of this theorem were
established. In particular we mention the $(p,q)$ Balian-Low theorem 
and the Amalgam Balian-Low theorem
(see discussion in Section \ref{related section}).
However, these versions
do not include any quantified
 estimate for the decay of the generator.

  A quantitative version of the classical uncertainty principle
  states that there exist constants $A,C>0$ such that for
   $g\in L^2(\R)$ 
\begin{equation} \label{a number}
        \int_{\abs{t} \geq R} \abs{g(t)}^2 \dif t + \int_{\abs{\xi} \geq L}  \abs{\hat{g}(\xi)}^2 \dif \xi \geq
       C \e^{-ARL}\norm{g}^2_{L^2(\R)},\:\:\:\:\:\:\:\:\:\:\forall R,L>0.
\end{equation}
(This follows from \cite{fuchs1964}.  See also Section \ref{open problems}.)
In light of the analogy between the
classical uncertainty principle and the classical Balian-Low
theorem, the following question was asked by K.~Grochenig: Is it
possible to obtain a version of \eqref{a number}
for generators
of Gabor Riesz bases?

Our main result is as follows.
\begin{theorem} \label{main result}
    Let $g \in L^2(\R)$. If $G(g)$ is a Riesz basis, then for $R,L\geq 1$ we have
    \begin{equation} \label{main estimate}
        \int_{\abs{x} \geq R} \abs{g(t)}^2 \dif t + \int_{\abs{\xi} \geq L}  \abs{\hat{g}(\xi)}^2 \dif \xi \geq
        \frac{C}{RL},
    \end{equation}
    where $C$ is a positive constant  depending only on the
    Riesz basis bounds.
%
%

This result is sharp in the sense that, for big enough $C$, the
lower bound cannot be replaced by $C\log (RL) /(RL)$.
\end{theorem}

Theorem \ref{main result} is strictly stronger than the classical
Balian-Low theorem as well as its $(p,q)$ version (see Sections
\ref{classical subsection} and \ref{pq subsection}). The Amalgam Balian-Low theorem is independent of
Theorem \ref{main result} (see Section \ref{amalgam subsection}). However, in Section
\ref{amalgam subsection}, we find a scale of uncertainty estimates that interpolate
Theorem \ref{main result} and the Amalgam Balian-Low theorem as
$(2,2)$ and $(1,\infty)$ endpoints.


The structure of the paper is as follows. In Section \ref{prelim section}, we
give some needed background and prove Lemma \ref{argument lemma}, which is the main step in the proof of Theorem \ref{main result}. We prove Theorem \ref{main result} in Section
\ref{main section}.. Finally,
as mentioned above, in Section \ref{related section}, we discuss the connection
between Theorem \ref{main result} and other Balian-Low type estimates.

\section{Preliminaries} \label{prelim section}

In this section, we recall the definitions of Riesz bases, the Zak
transform, and quasi-periodic functions.
Moreover, we obtain our key lemma, which quantifies the discontinuous
behavior of arguments of quasi-periodic functions.

\subsection{The Zak transform}
The following definition provides the main tool in the study of
lattice Gabor systems (see, for example, \cite[Chapter 8]{grochenig2001}).
\begin{definition}
    Let $g \in L^2(\R)$. The Zak transform of $g$ is given by
    \begin{equation*}
        Zg(x,y) = \sum_{k \in \Z} g(x-k) \e^{2\pi \im ky}, \qquad(x,y) \in \R^2.
    \end{equation*}
\end{definition}
 For $g \in L^2(\R)$ the function $Zg$ is 
quasi-periodic on $\R^2$. Namely, it satisfies
\begin{equation*}
    Zg(x,y+1) = Zg(x,y), \qquad \text{and} \qquad Zg(x+1,y) = \e^{2\pi \im y} Zg(x,y).
\end{equation*}
In particular, this implies that $Zg$ is uniquely determined by
its values on $[0,1]^2$. It is well-known that the Zak transform induces a unitary operator from
$L^2(\R)$ onto $L^2([0,1]^2)$.

We will use some basic properties of the Zak transform: The
Zak transform satisfies
\begin{equation}\label{zak x-y}
    Z\hat{g}(x,y) = \e^{2\pi \im xy} Zg(-y,x),
\end{equation}
where 
\begin{equation*}
	 \hat{g}(\xi) = \int_\R g(t) \e^{-2\pi \im \xi t} \dif t
\end{equation*}
is the Fourier transform for functions in $L^1(\R)$, which has   the usual extension to functions in $L^2(\R)$. Furthermore, we have
\begin{equation} \label{convolution}
    Z(g \ast \phi) = Zg \ast_x \phi,
\end{equation}
for any Schwarz class function $\phi$, where the convolution on
the right-hand side is taken with respect to the first variable.

\subsection{Riesz bases}

Let $H$ be a separable Hilbert space. A system $\{f_n\}$ in $H$ is
called a Riesz basis if it is the image of an orthonormal basis
under a bounded and invertible linear operator. This means that
$\{f_n\}$ is a Riesz basis if it is complete in $H$ and
\begin{equation}\label{r-b condition}
    A \sum|a_n|^2 \leq \Big\|\sum a_nf_n \Big\|^2 \leq B \sum|a_n|^2, \:\:\:\:\:\:\:\:\forall\{a_n\}\in \ell^2,
\end{equation}
where $A$ and $B$ are   positive constants. The biggest $A$ and
smallest $B$ for which \eqref{r-b condition} holds are called the
Riesz basis bounds.

It is well-known that if a Gabor system  
\begin{equation*}
	 G(g,a,b):= \{ \e^{2\pi
\im b  n t} g(t - ma)\}_{(m,n) \in \Z^2} 
\end{equation*}
 is a Riesz basis, then
$ab=1$ \cite{ramanathan_steger1995}. Although the results in this paper are stated for the case $a=b=1$, they also hold for all systems $G(g,a,b)$ with $ab=1$. This can be seen
 by appropriately dilating the generator function $g$.

The Zak transform allows one to characterize generators of
lattice Gabor Riesz bases. Indeed, it is easy to see that  $G(g)$ is
such a system if and only if
\begin{equation} \label{Zak bounds}
    A \leq \abs{Zg(x,y)}^2 \leq B, \qquad \forall (x,y) \in Q,
\end{equation}
where $A$ and $B$ are the Riesz basis bounds. This
characterization implies that in order to study the discontinuous
behavior of $Zg$, it is enough to understand the behavior of its
argument.

\subsection{Arguments of quasi-periodic functions}
It is known that an argument of a quasi-periodic function can
never be continuous (see, e.g., \cite[Lemma 8.4.2]{grochenig2001} and the references therein). We now establish our main lemma, which
is a stronger version of this statement.

First, we introduce some notation. Let $G$ be a quasi-periodic function on
$\R^2$ and $H$ be a branch of its argument, i.e.,
\begin{equation*}
     G(x,y) = \abs{G(x,y)} \e^{2\pi \im H(x,y)}.
\end{equation*}
Fix two integers $K, N\geq8$, and a point $(x,y) \in [0,1/K) \times
[0,1/N)$. Set
\begin{equation*}
    h_{i,j} = H\left(x + \frac{i}{K}, y + \frac{j}{N} \right), \qquad \forall (i,j) \in \Z^2.
\end{equation*}
\begin{lemma} \label{argument lemma}
 There exist two integers $0 \leq i < K$ and $0 \leq j < N$ such that either
    \begin{equation} \label{i jump}
        \abs{h_{i+1,j} - h_{i,j}} > \frac{1}{8} \mod 1
    \end{equation}
    or
    \begin{equation} \label{j jump}
        \abs{h_{i,j+1} - h_{i,j}} > \frac{1}{8} \mod 1.
    \end{equation}
\end{lemma}
\begin{proof}
    Assume that both  \eqref{i jump} and \eqref{j jump} do not hold, and
     choose a branch of the argument $H$ as follows:
    \begin{itemize}
    \item[-] First, fix the number $h_{0,0}$ and choose  $\{h_{i,0}\}_{i=1}^K$ to satisfy $\abs{h_{i+1,0} - h_{i,0}}\leq 1/8$  for $0 < i < K$.
    \item[]
    \item[-] Next, for each  $0 \leq i < K$, choose   $\{h_{i,j}\}_{j=1}^N$ 
    so that $\abs{h_{i,j+1} - h_{i,j}} \leq 1/8$ for $0\leq j<N$.
    \item[]
    \item[-] Finally, recall that the point $(x,y)$ was fixed and use the quasi-periodicity of $G$ to
    observe that $h_{K,0} = h_{0,0} + y + M$ for some integer $M$. Now, choose, for $0 < j \leq N$,
\begin{equation}\label{condition 3}
    h_{K,j}=h_{0,j}+\frac{j}{N}+y + M.
\end{equation}
	 This implies that $\abs{h_{K,j+1} - h_{K,j}} \leq 1/4$ for $0\leq j < N$.
    \end{itemize}
    All other values of the branch $H$ may be chosen arbitrarily.

     We claim that with these choices we also have
    \begin{equation} \label{mother}
        \abs{h_{i+1, N}  - h_{i,N}} \leq \frac{1}{8}, \qquad 0 \leq i < K.
    \end{equation}
 Indeed, to prove \eqref{mother}, first note that, since $G$ is a quasi-periodic function, this inequality holds modulo 1. Now we may use a
     recursive argument: As $\abs{h_{i+1,0} - h_{i,0}}\leq 1/8$, it follows, by comparing with these two
     terms, that
     $\abs{h_{i+1, 1} - h_{i,1}} \leq 1/2$. But, since $\abs{h_{i+1,1} -h_{i,1}} \leq 1/8$ modulo $1$, it is clear that this inequality also holds
     without taking modulo $1$.
     Repeating this argument $N$ times, one arrives at
     \eqref{mother}.

By the periodicity of $G$ in the $y$-variable, there exist
integers $L_i$ so that $h_{i,N} = h_{i, 0} + L_i$. By
\eqref{mother},
    we have that $\abs{L_{i+1} - L_i} \leq 1/4$. Hence,  all $L_i$ are equal. In particular, $L_K=L_0$.

     On the other hand,
    condition \eqref{condition 3} implies that the
sequence $\{b_j\}:=\{h_{K,j}-h_{0,j}\}$ satisfies
\begin{equation*}
b_{j+1}-b_j=\frac{1}{N},
\end{equation*}
whence $L_{K}-L_{0}=b_N-b_0=1$. This gives the required
contradiction.
\end{proof}

\section{Proof of Theorem \ref{main result}} \label{main section}

In this section, we obtain some consequences of Lemma \ref{argument lemma} which we then   use to prove our main result.

 \subsection{Bounded quasi-periodic functions}

A basic measure theoretic argument   yields the following
consequence of Lemma \ref{argument lemma}.
\begin{lemma}\label{corollary}
    Fix $D>0$. There exists a positive constant
    $\delta=\delta(D)$ such that, given any two integers $K,N\geq8$ and any quasi-periodic function $G$
    with
    $\abs{G}\geq D$ almost everywhere, one can find
 a set $S\subseteq [0,1]^2$ of measure
    at least $1/NK$ such that all $(x,y) \in S$ satisfy either
    \begin{equation}\label{coro-cond-1}
        \Bigabs{ G  ( x+{\frac{1}{K}}, y )-G  ( x, y )} \geq \delta,
    \end{equation}
    or
    \begin{equation}\label{coro-cond-2}
        \Bigabs{ G  ( x, y+\frac{1}{N})- G  ( x, y)} \geq \delta.
    \end{equation}
\end{lemma}

Next, we use Lemma \ref{corollary} to evaluate the difference
between the Zak transform of a function and the Zak transform of its
convolution with some smooth kernel.
\begin{lemma} \label{average lemma}

    Fix $m, M >0$. There exist positive constants
    $\delta=\delta(m)$ and $\eta=\eta(m,M)$ such that, given any pair of Schwarz class functions $\phi,
    \psi$ on $\R$
    and any $g\in L^2(\R)$
    with
    $m<\abs{Zg}<M$ almost everywhere, one can find
 a set $S\subseteq [0,1]^2$ of measure
    at least $\eta/(1+\int|\phi'|)(1+\int|\psi'|)$ such that all $(x,y) \in S$ satisfy either
    \begin{equation}\label{cor2-cond-1}
        \abs{Zg(x,y) - Z(g \ast \phi)(x,y)}\geq \delta,
    \end{equation}
    or
    \begin{equation}\label{cor2-cond-2}
     \abs{Z\hat{g}(x,y) - Z(\hat{g} \ast \psi)(x,y)} \geq \delta.
    \end{equation}
\end{lemma}
\begin{proof}
Observe that by \eqref{convolution}, for all $K>0$,
we have
\begin{equation}\label{lipschitz}
    \abs{ Z(g \ast \phi)(x+1/K,y) - Z (g \ast \phi)(x,y) }  \leq
    \frac{ M}{K} \int_\R \abs{\phi'(\tau)} \dif \tau.
\end{equation}

For $G = Zg$ and $D=m$, let $\delta_1>0$ be the constant from Lemma
\ref{corollary}. Choose the smallest integers $N,K\geq 8$ that satisfy
\begin{equation*}
 K\geq  \frac{2M}{\delta_1} \int
 \abs{\phi'}  \qquad \text{and} \qquad N \geq  \frac{4\pi M}{\delta_1}\left( \int
 \abs{\psi'} +1\right),
\end{equation*}
and let $S_1$ be the set described in Lemma \ref{corollary}.

If $(x,y)\in S_1$ satisfies \eqref{coro-cond-1}, then, in light of
\eqref{lipschitz}, we find that either $(x,y)$ or $(x+1/K,y)$
satisfy \eqref{cor2-cond-1} with $\delta=\delta_1/4$.

On the other hand, if $(x,y)\in S_1$ satisfies \eqref{coro-cond-2},
then \eqref{zak x-y} and the quasi-periodicity of $Zg$ imply that
\begin{equation*}
     \abs{Z\hat{g}(1-y - 1/N, x) - Z\hat{g}(1-y,x)}
     \geq \delta_1 -\frac{2\pi}{N}M \geq \frac{\delta_1}{2}.
\end{equation*}
Applying \eqref{lipschitz} once again, we find that either
$(1-y,x)$ or $(1-y-1/N, x)$ satisfy \eqref{cor2-cond-2} with
$\delta=\delta_1/8$.

 In this way we   find a set $S$ of measure at least $|S_1|/4$
 with the required properties.

 (It may happen that $(x+1/K,y)$ or $(1-y-1/N,x)$ fall outside of
$[0,1]^2$. However, as the functions discussed are quasi-periodic,
the conclusion remains true if $x+1/K$ and $1-y-1/N$ are taken
modulo 1).
\end{proof}

\subsection{Proof of Theorem \ref{main result}}
\begin{proof}
Let $g \in L^2(\R)$ be such that $G(g)$ is a Riesz basis with
bounds $A$ and $B$. By \eqref{Zak bounds}, we have $m \leq
\abs{Zg} \leq M$ with $m=\sqrt{A}$ and $M=\sqrt{B}$.

Fix a Schwarz class function $\rho$ such that $\hat{\rho}$ is symmetric, satisfies
$|\hat{\rho}|\leq 1$ and
\begin{equation*}
     \hat{\rho}(\xi) = \left\{ \begin{split}  1 & \quad \text{if} \quad \abs{\xi} \leq 1,
     \\
      0 & \quad \text{if} \quad \abs{\xi} \geq 2. \end{split} \right.
\end{equation*}
Given $R,L\geq 1$, we apply Lemma \ref{average lemma} with
$\phi(t) = R \rho(Rt)$ and $\psi(t)=L\rho(Lt)$. We find that, for
some constants $C=C(A,B)$ and $\delta=\delta(A,B)$, there exists a
set $S\subseteq [0,1]^2$ with measure at least $C/RL$ such
that all $(x,y)\in S$ satisfy
\begin{equation*}
     \delta^2 \leq  \abs{Z\hat{g} - Z(\hat{g} \ast \phi) }^2 +  \abs{Zg - Z(g \ast \psi)}^2.
\end{equation*}
From this and \eqref{zak x-y} it follows that
\begin{align*}
     \frac{C\delta^2}{RL} &\leq \iint_{[0,1]^2}  \abs{Z\hat{g} - Z(\hat{g}\ast \phi)}^2  + \abs{Zg - Z({g} \ast \psi)}^2  \dif x \dif y \\
     &= \iint_{[0,1]^2}  \abs{Z\big({g}(1 -   \hat{\phi})\big)}^2 +   \abs{Z(\hat{g}\big(1 -  \hat{\psi})\big)}^2 \dif x \dif y.
\end{align*}
Since the Zak transform is a unitary operator from $L^2(\R)$ to
$L^2([0,1]^2)$, this gives
\begin{align*}
    \frac{C_1}{RL} 
    &\leq   \int_{\R}  \abs{g(t)}^2 \abs{1 - \hat{\phi}(t)}^2 \dif t +  \int_{\R}  \abs{\hat{g}(\xi)}^2 \abs{1 - \hat{\psi}(\xi)}^2 \dif \xi \\
    &\leq     \int_{\abs{t}>R}  \abs{{g}(t)}^2   \dif t + \int_{\abs{\xi}>L}  \abs{\hat{g}(\xi)}^2   \dif \xi,
\end{align*}
%
where $C_1=C\delta^2$ is a positive constant depending only on $A$
and $B$.

To see that the result is sharp, we use the function $g$ defined in \cite[Definition 2.3]{benedetto_czaja_gadzinski_powell2003}.
On the one hand, this function satisfies $\abs{Zg} =1$ for all
$(x,y) \in \R^2$, and therefore the system $G(g)$ is an
orthonormal basis. On the other hand, by following the proofs of
Theorems 3.4 and 3.10 in the same reference, it follows that
\begin{equation*}
    \int_{\abs{t}>R} \abs{g(t)}^2 \dif t + \int_{\abs{\xi} >L}   \abs{\hat{g}(\xi)}^2 \dif \xi \leq \frac{1}{R^2} + \frac{\log L}{L^2}.
\end{equation*}
This ends the proof.
\end{proof}

\section{ Theorem \ref{main result} in comparison to related results} \label{related section}
In this section, we discuss the connection between Theorem \ref{main result}, the
classical Balian-Low theorem and other results in this field. In
particular,  in Section \ref{pq subsection} we study the connection to the $(p,q)$
Balian-Low theorem and in Section \ref{amalgam subsection} the connection to the
Amalgam Balian-Low theorem.
\subsection{The classical Balian-Low theorem} \label{classical subsection}
We begin by observing that Theorem 1 implies the Balian-Low
theorem.
%
%
Indeed, it is enough to notice that if $g\in L^2(\R)$ satisfies
\eqref{main estimate}, then for every $R>0$ we have
\begin{equation} \label{calculation}
\begin{split}
C &\leq R^2\int_{|t|>R}|g(t)|^2\dif t+R^2\int_{|\xi|>R}|\hat{g}(\xi)|^2\dif \xi \\
& \leq \int_{|t|>R}|t|^2|g(t)|^2\dif
t+\int_{|\xi|>R}|\xi|^2|\hat{g}(\xi)|^2\dif \xi,
\end{split}
\end{equation}
whence at least one of the integrals in \eqref{clasical
blt condition}   diverges.

This proves the first half of the following proposition.

\begin{proposition} \label{th1 better}  If a function in $L^2(\R)$ satisfies condition \eqref{main estimate},
then it also satisfies condition \eqref{clasical blt condition}.
Moreover, there exists a function $g\in L^2(\R)$ that satisfies condition \eqref{clasical blt condition} but does not
satisfy condition \eqref{main estimate}.
\end{proposition}
\begin{proof}
%
Throughout this proof, we let $C$ denote
different constants which may change from line to line. Let
\begin{equation*}
 {g}(t)=\sum_{k=1}^{\infty}\frac{\sin^2(t-2^k)}{k^{1/2}2^k(t-2^k)^2},
\end{equation*}
and note that its Fourier transform is supported on $[-2,2]$. The
function ${g}$ satisfies
\begin{equation*}
\int_\R |t|^2| {g}(t)|^2\dif t\geq
\sum_{k=1}^{\infty}\frac{1}{k}\int_{2^k}^{2^k+1}\frac{\sin^4(t-2^k)}{(t-2^k)^4}\dif t=\infty.
 \end{equation*}
  Given a positive integer $n$, set
$R_n=[2^n+2^{(n+1)}]/2$. To estimate
\begin{equation*}
    I:=\int_{|t|>R_n}| {g}(t)|^2\dif t,
\end{equation*}
write $ {g}=\sum_{k=1}^n+\sum_{k=n+1}^{\infty}$ and
correspondingly $I<2I_1+2I_2$. By the Cauchy-Schwartz inequality,
we have
\begin{equation*}
I_1:=\int_{|t|>R_n}\Bigabs{\sum_{k=1}^{n}\frac{1}{k^{1/2}2^k(t-2^k)^2}}^2\dif t
\leq C\frac{n}{2^{3n}},
\end{equation*}
which is smaller than $C/R_n^2$ whenever $n$ is big enough. We
turn to estimating $I_2$. To this end let $w$ be the Fourier
transform of ${\sin^2t/t}$. We have
\[
I_2 
\leq \int_{\R}\Bigabs{\sum_{k=n+1}^{\infty}\frac{1}{k^{1/2}2^k}e^{i2^k\xi}w(\xi)}^2\dif t \leq
C\frac{1}{n2^{2n}},
\]
which is smaller than $C/R_n^2$ whenever $n$ is big enough.
It follows that the function ${g}$ satisfies condition
\eqref{clasical blt condition}, while condition \eqref{main
estimate} is not fulfilled.
\end{proof}
\begin{remark} \label{first remark}
    In fact, the function $g + \hat{g}$ is an example of a function that does not satisfy condition \eqref{main estimate}, but for which  both integrals in \eqref{clasical blt condition} diverge.
\end{remark}

To better understand the relation between Theorem \ref{main result} and the
classical Balian-Low theorem, we consider the symmetric case $L=R$. In  this case,  $g \in L^2(\R)$ satisfies condition \eqref{main estimate} of Theorem \ref{main result} if and only if
\begin{equation} \label{hello}
\liminf_{R\rightarrow\infty}R^2 \left(\int_{|t|>R}|g(t)|^2\dif t +
 \int_{|\xi|>R}|\hat{g}(\xi)|^2\dif \xi \right)>C.
\end{equation}

The classical Balian-Low theorem is
related to similar statements, with $\limsup$ in place of $\liminf$.  Indeed, the calculation \eqref{calculation} implies that  if a function $g \in L^2(\R)$ satisfies   condition \eqref{hello} with this modification, then  it also  satisfies   condition \eqref{clasical blt condition}.

As a converse to this statement, it is easy to show the following. 
If $g\in L^2(\R)$  satisfies condition \eqref{clasical blt
condition} then, for every $\epsilon>0$, it holds that
\begin{equation*}
\limsup_{R\rightarrow\infty}R^{2+\epsilon} \left( \int_{|t|>R}|g(t)|^2\dif t 
+   \int_{|\xi|>R}|\hat{g}(\xi)|^2\dif \xi \right) =\infty.
\end{equation*}
This is sharp in the sense that removing the $\epsilon$ may give a limit equal to $0$. 

\subsection{ The $(p,q)$ Balian-Low theorem} \label{pq subsection}

Fix $1\leq p\leq 2$ and let $ q $ be such that $1/p+1/q=1$.
The $(p,q)$ Balian-Low theorem states that if $g\in
L^2(\R)$ generates a Gabor Riesz basis on the integer lattice, then
\begin{equation}\label{p-q}
\int|t|^p|g(t)|^2\dif t=\infty\:\:\:\:\:\text{or}\:\:\:\:\:\int|\xi|^q|\hat{g}(\xi)|^2\dif \xi=\infty.
\end{equation}
For  exponents $p+\epsilon$ and $q+\epsilon$ in place of $p$ and $q$,
the original proof follows by combining \cite[Theorem 4.4]{feichtinger_grochenig1997} with \cite[Theorem 1]{grochenig1996}.
This was improved to \eqref{p-q} in \cite{gautam2008} using methods from the theory of VMO functions.

The following proposition implies that the $(p,q)$ Balian-Low
theorem follows from Theorem \ref{main result}.

\begin{proposition}
If a function in $L^2(\R)$ satisfies condition \eqref{main
estimate}, then it also satisfies condition \eqref{p-q}.
Moreover, there exists a function $g\in L^2(\R)$ that satisfies condition \eqref{p-q},  but does not
satisfy condition \eqref{main estimate}.
\end{proposition}

\begin{proof}
Assume that  $g\in L^2(\R)$ satisfies condition
\eqref{main estimate}. For $R>0$ and $L=R^{p-1}$, this
 implies that
\begin{equation*}
C\leq\int_{|t|>R}|t|^{p}|g(t)|^2\dif t+\int_{|\xi|>R^{p-1}}|\xi|^{q}|\hat{g}(\xi)|^2\dif \xi,
\end{equation*}
which means that at least one of the integrals in \eqref{p-q} diverges.

The second part of the statement follows from Remark \ref{first
remark}.
\end{proof}


\subsection{ The Amalgam Balian-Low theorem} \label{amalgam subsection}

The Amalgam Balian-Low Theorem \cite[Corollary 7.5.3]{heil1990} states that if $g\in L^2(\R)$
generates a Gabor Riesz Basis on the integer lattice, then
\begin{equation}\label{amalgam}
\sum_n\|g(t)\|_{L^{\infty}[n,n+1]}=\infty
\:\:\:\:\:\textrm{or}\:\:\:\:\:\sum_n\|\hat{g}(\xi)\|_{L^{\infty}[n,n+1]}=\infty.
\end{equation}
This theorem is known to be independent of the classical
Balian-Low theorem \cite[Examples 3.3 and 3.4]{benedetto_heil_walnut1995}.
A simple calculation shows that the same examples also yield
that the Amalgam Balian-Low theorem  is independent of  Theorem \ref{main result}:


%
\begin{proposition}
There exists a function $g\in L^2(\R)$ that satisfies condition
\eqref{main estimate} but does not satisfy condition \eqref{amalgam}.
On the other hand, there exists a function $g\in L^2(\R)$ that
satisfies condition \eqref{amalgam},  but does not satisfy
condition \eqref{main estimate}.
\end{proposition}

We proceed to show that the Amalgam Balian-Low theorem and Theorem \ref{main result} are
  end-points of a more general Balian-Low type theorem.

\begin{theorem}
Let $1 \leq p \leq 2$ and   $q$ be such that $1/p + 1/q = 1$.
If $G(g)$ is a Riesz basis, then for $R,L\geq 1$ we have
\begin{equation*}
     \frac{C}{(RL)^{p/q}}<  \sum_{\abs{k}>R} \norm{f}_{L^q(k,k+1)}^p+\sum_{\abs{k}>L} \norm{\hat{f}}_{L^q(k,k+1)}^p ,
\end{equation*}
    where $C$ is a positive constant  depending only on the
    Riesz basis bounds.
\end{theorem}
%
Note that for $p=q=2$ we get Theorem \ref{main result}, while for $p=1$ and $q=\infty$ we get the Amalgam
Balian-Low theorem.
\begin{proof}
It is easy to verify that the Zak transform satisfies
\begin{equation*}
    \norm{Zg}_{L^2(Q)}^2 = \norm{g}_{L^2(\R)}^2= \sum_k \norm{g}_{L^2(k,k+1)}^2,
    \end{equation*}
    and
    \begin{equation*}
    \norm{Zg}_{L^\infty(Q)}  \leq \sum_k \norm{g}_{L^\infty(k,k+1)}.
\end{equation*}

By the method of complex interpolation for vector valued sequence
spaces (see, e.g., \cite[Theorem 5.1.2]{bergh_lofstrom1976}), this implies that
\begin{equation}\label{interpolation}
 \norm{Zg}_{L^q(Q)}^p  \leq \sum_k \norm{g}_{L^q(k,k+1)}^p.
\end{equation}

We first prove the theorem for $1<p<2$. Following the same
technique as in the proof of Theorem \ref{main result}, we use Lemma
\ref{average lemma} to find that

\begin{align*}
     \frac{C\delta^q}{RL} & \leq \iint_Q \abs{Z\hat{g} - Z(\hat{g} \ast \phi)}^q +  \abs{Zg - Z( g \ast \psi)}^q \dif x \dif y \\
     &=  \iint_Q \abs{Z\big( g(1-  \hat{\phi}) \big)}^q +   \abs{Z\big( \hat{g}(1-    \hat{\psi})\big)}^q \dif x \dif y.
\end{align*}
By \eqref{interpolation}, this is smaller than
\begin{multline*}
    \left(\sum_{k} \norm{g(1 - \hat{\phi})}_{L^q(k,k+1)}^p \right)^{q/p}+ \left(\sum_k \norm{\hat{g}(1- \hat{\psi})}_{L^q(k,k+1)}^p\right)^{q/p}
    \\
    \leq
    \left(\sum_{\abs{k}>R} \norm{g }_{L^q(k,k+1)}^p \right)^{q/p}+ \left(\sum_{\abs{k}>L} \norm{\hat{g} }_{L^q(k,k+1)}^p\right)^{q/p}.
\end{multline*}
The desired inequality now follows from the fact that $\abs{x}^\beta +
\abs{y}^\beta \leq C (\abs{x}+\abs{y})^\beta$ for all $\beta\geq 0$, with $C>0$
depending only on $\beta$.

The case $p=1$ is proved in much the same way.
\end{proof}

\section{Open problem} \label{open problems}
A much more general version of \eqref{a number} was obtained by F. Nazarov \cite{nazarov1993}:
 There exist constants $A,C>0$ such that for $g\in L^2(\R)$ and   any two
sets $S,K\subset \R$ of finite measure the following inequality holds
\begin{equation*}
        \int_{\R\setminus S} \abs{g(x)}^2 \dif x + \int_{\R\setminus K}  \abs{\hat{g}(\xi)}^2 \dif \xi \geq
        C\e^{-A|S||K|}\norm{g}^2_{L^2(\R)}.
\end{equation*}
Does this theorem have a version for functions $g$ that generate
  Gabor Riesz bases on the integer lattice?

\section*{Acknowledgements} The authors wish to thank A.~Aleman for inviting the first
author to Lund university and M.~Sodin for inviting the second
author to Tel-Aviv university. These visits provided us with the
best setting possible for completing this project.

\bibliographystyle{amsplain}
\def\cprime{$'$} \def\cprime{$'$} \def\cprime{$'$}
\providecommand{\bysame}{\leavevmode\hbox to3em{\hrulefill}\thinspace}
\providecommand{\MR}{\relax\ifhmode\unskip\space\fi MR }
\providecommand{\MRhref}[2]{%
  \href{http://www.ams.org/mathscinet-getitem?mr=#1}{#2}
}
\providecommand{\href}[2]{#2}

\end{document}